\documentclass[12pt, reqno]{amsart}
\usepackage{amsmath} 
\usepackage{amsthm}
\usepackage{amssymb}
\usepackage{graphics}
\usepackage{latexsym}
\usepackage{color}

\numberwithin{equation}{section}
\newtheorem{thm}{Theorem}[section]
\newtheorem{prop}[thm]{Proposition}
\newtheorem{lemm}[thm]{Lemma}
\newtheorem{cor}[thm]{Corollary}
\newtheorem{defn}[thm]{Definition}
\newtheorem{rem}[thm]{Remark}
\newcommand{\BBB}{\mathbb}
\newcommand{\R}{{\BBB R}}
\newcommand{\Z}{{\BBB Z}}

\newcommand{\N}{{\BBB N}}
\newcommand{\C}{{\BBB C}}
\newcommand{\ZZ}{\mathcal{Z}}
\newcommand{\ee}{\mbox{\boldmath $1$}}

\newcommand{\F}{\mathcal{F}}
\newcommand{\kuuhaku}{\text{}}
\title[WP and scattering for 4NLS]{Well-posedness  and scattering for 
fourth order nonlinear Schr\"odinger type equations at the scaling critical regularity
}
\author[H. Hirayama]{Hiroyuki Hirayama}
\author[M. Okamoto]{Mamoru Okamoto}
\address[H. Hirayama]{Graduate School of Mathematics, Nagoya University,
Chikusa-ku, Nagoya, 464-8602, Japan}
\address[M. Okamoto]{Department of Mathematics, Institute of Engineering, Academic Assembly, Shinshu University, 4-17-1 Wakasato, Nagano City 380-8553, Japan}
\email[H. Hirayama]{m08035f@math.nagoya-u.ac.jp}
\email[M. Okamoto]{m\_okamoto@shinshu-u.ac.jp}

\date{}
\begin{document}
\maketitle
\begin{abstract}
In the present paper, we consider the Cauchy problem of  fourth order nonlinear 
Schr\"odinger type equations with a derivative nonlinearity. 
In one dimensional case, we prove that the fourth order nonlinear Schr\"odinger equation with the derivative quartic nonlinearity $\partial _x (\overline{u}^4)$ is the small data global in time well-posed and scattering to a free solution.
Furthermore, we show that the same result holds for the $d \ge 2$ and derivative polynomial type nonlinearity, for example $|\nabla | (u^m)$ with $(m-1)d \ge 4$.
\\

\end{abstract}
%

\section{Introduction\label{intro}}
We consider the Cauchy problem of the fourth order nonlinear Schr\"odinger type equations:
\begin{equation}\label{D4NLS}
\begin{cases}
\displaystyle (i\partial_{t}+\Delta ^2)u=\partial P_{m}(u,\overline{u}),\hspace{2ex}(t,x)\in (0,\infty )\times \R^{d} \\
u(0,x)=u_{0}(x),\hspace{2ex}x\in \R^{d}
\end{cases}
\end{equation}
where $m\in \N$, $m\geq 2$, $P_{m}$ is a polynomial which is written by
\[
P_{m}(f,g)=\sum_{\substack{\alpha ,\beta \in \Z_{\geq 0}\\ \alpha +\beta=m}}f^{\alpha}g^{\beta}, 
\]
$\partial$ is a first order derivative with respect to the spatial variable, for example a linear combination of  
$\frac{\partial}{\partial x_1} , \, \dots , \, \frac{\partial}{\partial x_d}$ or $|\nabla |= \mathcal{F}^{-1}[|\xi | \mathcal{F}]$
and the unknown function $u$ is $\C$-valued. 
The fourth order Schr\"{o}dinger equation with $P_{m}(u,\overline{u})=|u|^{m-1}u$ appears in the study of deep water wave dynamics \cite{Dysthe}, solitary waves \cite{Karpman}, \cite{KS}, vortex filaments \cite{Fukumoto}, and so on.
The equation (\ref{D4NLS}) is invariant under the following scaling transformation:
\[
u_{\lambda}(t,x)=\lambda^{-3/(m-1)}u(\lambda^{-4}t,\lambda^{-1}x), 
\]
and the scaling critical regularity is $s_{c}=d/2-3/(m-1)$. 
The aim of this paper is to prove the well-posedness and the scattering for the solution of (\ref{D4NLS}) 
in the scaling critical Sobolev space.

There are many results for the fourth order nonlinear Schr\"{o}dinger equation 
with derivative nonlinearities (see \cite{S1}, \cite{S2}, \cite{HJ1}, \cite{HHW}, \cite{HHW2}, \cite{HJ3}, \cite{S3}, \cite{HJ2}, \cite{Y12}, \cite{HN15_1}, \cite{HN15_2}, and references cited therein).
Especially, the one dimensional case is well studied.
Wang (\cite{Y12}) considered (\ref{D4NLS}) for the case $d=1$, $m=2l+1$, $l\ge 2$, $P_{2l+1}(u,\overline{u})=|u|^{2l}u$ 
and proved the small data global in time well-posedness for $s=s_{c}$ by using Kato type smoothing effect. 
But he did not treat the cubic case.
Actually, a technical difficulty appears in this case (see Theorem \ref{notC3} below).

Hayashi and Naumkin (\cite{HN15_1}) considered (\ref{D4NLS}) for $d=1$ with the power type nonlineality $\partial_{x}(|u|^{\rho -1}u)$ ($\rho >4$) 
and proved the global existence of the solution and the scattering in the weighted Sobolev space.
Moreover, they (\cite{HN15_2}) also proved that the large time asymptotics is determined by the self similar solution in the case $\rho =4$.
Therefore, derivative quartic nonlinearity in the one spatial dimension is the critical in the sense of the asymptotic behavior of the solution.

We firstly focus on the quartic nonlinearity $\partial _x (\overline{u}^4)$ in one space dimension.
Since this nonlinearity has some good structure, the global solution scatters to a free solution in the scaling critical Sobolev space.
Our argument does not apply to \eqref{D4NLS} with $P (u,\overline{u}) = |u|^3 u$ because we rely on the Fourier restriction norm method.
Now, we give the first results in this paper. 
For a Banach space $H$ and $r>0$, we define $B_r(H):=\{ f\in H \,|\, \|f\|_H \le r \}$. 
\begin{thm}\label{wellposed_1}
Let $d=1$, $m=4$ and $P_{4}(u,\overline{u})=\overline{u}^{4}$. Then the equation {\rm (\ref{D4NLS})} is globally well-posed for small data in $\dot{H}^{-1/2}$. 
More precisely, there exists $r>0$ such that for any $T>0$ and all initial data $u_{0}\in B_{r}(\dot{H}^{-1/2})$, there exists a solution
\[
u\in \dot{Z}_{r}^{-1/2}([0,T))\subset C([0,T );\dot{H}^{-1/2})
\]
of {\rm (\ref{D4NLS})} on $(0, T )$. 
Such solution is unique in $\dot{Z}_{r}^{-1/2}([0,T))$ which is a closed subset of $\dot{Z}^{-1/2}([0,T))$ {\rm (see Definition~\ref{YZ_space} and (\ref{Zr_norm}))}. 
Moreover, the flow map
\[
S^{+}_{T}:B_{r}(\dot{H}^{-1/2})\ni u_{0}\mapsto u\in \dot{Z}^{-1/2}([0,T))
\]
is Lipschitz continuous. 
\end{thm}
\begin{rem}
We note that $s=-1/2$ is the scaling critical exponent of (\ref{D4NLS}) for $d=1$, $m=4$. 
\end{rem}
\begin{cor}\label{sccat}
Let $r>0$ be as in Theorem~\ref{wellposed_1}. 
For all $u_{0}\in B_{r}(\dot{H}^{-1/2})$, there exists a solution 
$u\in C([0,\infty );\dot{H}^{s_{c}})$ of (\ref{D4NLS}) on $(0,\infty )$ and the solution scatters in $\dot{H}^{-1/2}$. 
More precisely, there exists 
$u^{+}\in \dot{H}^{-1/2}$ 
such that 
\[
u(t)-e^{it\Delta^2}u^{+}
\rightarrow 0
\ {\rm in}\ \dot{H}^{-1/2}\ {\rm as}\ t\rightarrow + \infty. 
\]
\end{cor}

Moreover, we obtain the large data local in time well-posedness in the scaling critical Sobolev space.
To state the result, we put
\[
B_{\delta ,R} (H^s) := \{ u_0 \in H^s | \ u_0=v_0+w_0 , \, \| v_0 \| _{\dot{H}^{-1/2}} < \delta, \, \| w_0 \| _{L^2} <R \}
\]
for $s<0$.

\begin{thm} \label{large-wp}
Let $d=1$, $m=4$ and $P_{4}(u,\overline{u})=\overline{u}^{4}$.  
Then the equation {\rm (\ref{D4NLS})} is locally in time well-posed in $H^{-1/2}$. 
More precisely,
there exists $\delta >0$ such that for all $R \ge \delta$ and $u_0 \in B_{\delta ,R} (H^{-1/2})$ there exists a solution
\[
u \in Z^{-1/2}([0,T]) \subset C([0,T); H^{-1/2})
\]
for $T=\delta ^{8} R^{-8}$ of \eqref{D4NLS}.

Furthermore, the same statement remains valid if we replace $H^{-1/2}$ by $\dot{H}^{-1/2}$ as well as $Z^{-1/2}([0,T])$ by $\dot{Z}^{-1/2}([0,T])$.
\end{thm}

\begin{rem}
For $s>-1/2$, the local in time well-posedness in $H^s$ follows from the usual Fourier restriction norm method, which covers for all initial data in $H^s$.
It however is not of very much interest.
On the other hand, since we focus on the scaling critical cases, which is the negative regularity, we have to impose that the $\dot{H}^{-1/2}$ part of  initial data is small.
But, Theorem \ref{large-wp} is a large data result because the $L^2$ part is not restricted.
\end{rem}

The main tools of the proof are the $U^{p}$ space and $V^{p}$ space which are applied to prove 
the well-posedness and the scattering for KP-II equation at the scaling critical regularity by Hadac, Herr and Koch (\cite{HHK09}, \cite{HHK10}).

We also consider the one dimensional cubic case and the high dimensional cases. 
The second result in this paper is as follows.
\begin{thm}\label{wellposed_2}
{\rm (i)}\ Let $d=1$ and $m=3$. Then the equation {\rm (\ref{D4NLS})} is locally well-posed  in $H^{s}$ for $s\ge 0$. \\
{\rm (ii)}\ Let $d\geq 2$ and $(m-1)d\geq 4$. Then the equation {\rm (\ref{D4NLS})}
 is globally well-posed for small data in $\dot{H}^{s_{c}}$ (or $H^{s}$ for $s\ge s_{c}$)
 and the solution scatters in $\dot{H}^{s_{c}}$ (or $H^{s}$ for $s\ge s_{c}$).
\end{thm}

The smoothing effect of the linear part recovers derivative in higher dimensional case.
Therefore, we do not use the $U^p$ and $V^p$ type spaces.
More precisely, to establish Theorem \ref{wellposed_2}, we only use the Strichartz estimates
and get the solution in $C([0,T);H^{s_c})\cap L^{p_m}([0,T); W^{q_m,s_{c}+1/(m-1)})$ 
with $p_m =2(m-1)$, $q_m =2(m-1)d/\{ (m-1)d-2\}$.
Accordingly, the scattering follows from a standard argument.
Since the condition $(m-1)d\geq 4$ is equivalent to $s_{c}+1/(m-1)\ge 0$, 
the solution space $L^{p_m}([0,T); W^{q_m,s_{c}+1/(m-1)})$ has nonnegative regularity even if the data belongs to $H^{s_{c}}$ with $-1/(m-1)\le s_c <0$. 
Our proof of Thorem~\ref{wellposed_2} {\rm (ii)} cannot applied for $d=1$ 
since the Schr\"odingier admissible $(a,b)$ in {\rm (\ref{admissible_ab})} does not exist. 
\begin{rem}
For the case $d=1$, $m=4$ and $P_{4}(u,\overline{u})\ne \overline{u}^{4}$, 
we can obtain the local in time well-posedness of {\rm (\ref{D4NLS})} in $H^{s}$ for $s\ge 0$ 
by the same way of the proof of Theorem~\ref{wellposed_2}. 
Actually, we can get the solution in $C([0,T];H^s)\cap L^4 ([0,T];W^{s+1/2,\infty })$ 
for $s\ge 0$ by using the iteration argument  
since the fractional Leibnitz rule (see \cite{CW91}) and the H\"older inequality imply
\[
\left\| |\nabla |^{s+\frac{1}{2}}\prod_{j=1}^{4}u_j \right\|_{L^{4/3}_{t}([0,T);L_{x}^{1})}
\lesssim T^{1/4}\| |\nabla |^{s+\frac{1}{2}}u_1 \|_{L^{4}_{t}L_{x}^{\infty}}\| u_2 \|_{L^{4}_{t}L_{x}^{\infty}}
\| u_3 \|_{L^{\infty}_{t}L_{x}^{2}}\| u_4 \|_{L^{\infty}_{t}L_{x}^{2}}.
\]
\end{rem}

We give a remark on our problem, which shows that the standard iteration argument does not work.
\begin{thm}\label{notC3}
{\rm (i)}\ Let $d=1$, $m=3$, $s<0$ and $P_{3}(u,\overline{u})=|u|^{2}u$. Then the flow map of {\rm (\ref{D4NLS})} from $H^s$ to $C(\R ; H^s)$ is not smooth. \\
{\rm (ii)}\ Let $m\ge 2$, $s<s_c$ and $\partial =|\nabla |$ or $\frac{\partial}{\partial x_k}$ for some $1\le k \le d$. 
Then the flow map of {\rm (\ref{D4NLS})} from $H^s$ to $C(\R ; H^s)$ is not smooth.
\end{thm}

More precisely, we prove that the flow map is not $C^3$ if $d=1$, $m=3$, $s<0$ and $P_{3}(u,\overline{u})=|u|^{2}u$ or $C^m$ if $d \ge 1$, $m \ge 2$, and $s<s_c$.
It leads that the standard iteration argument fails, because the flow map is smooth if it works.
Of course, there is a gap between ill-posedness and absence of a smooth flow map.

Since the resonance appears in the case $d=1$, $m=3$ and $P_{3}(u,\overline{u})=|u|^{2}u$, there exists an irregular flow map even for the subcritical Sobolev regularity.

\kuuhaku \\
\noindent {\bf Notation.} 
We denote the spatial Fourier transform by\ \ $\widehat{\cdot}$\ \ or $\F_{x}$, 
the Fourier transform in time by $\F_{t}$ and the Fourier transform in all variables by\ \ $\widetilde{\cdot}$\ \ or $\F_{tx}$. 
The free evolution $S(t):=e^{it\Delta^{2}}$ is given as a Fourier multiplier
\[
\F_{x}[S(t)f](\xi )=e^{-it|\xi |^{4}}\widehat{f}(\xi ). 
\]
We will use $A\lesssim B$ to denote an estimate of the form $A \le CB$ for some constant $C$ and write $A \sim B$ to mean $A \lesssim B$ and $B \lesssim A$. 
We will use the convention that capital letters denote dyadic numbers, e.g. $N=2^{n}$ for $n\in \Z$ and for a dyadic summation we write
$\sum_{N}a_{N}:=\sum_{n\in \Z}a_{2^{n}}$ and $\sum_{N\geq M}a_{N}:=\sum_{n\in \Z, 2^{n}\geq M}a_{2^{n}}$ for brevity. 
Let $\chi \in C^{\infty}_{0}((-2,2))$ be an even, non-negative function such that $\chi (t)=1$ for $|t|\leq 1$. 
We define $\psi (t):=\chi (t)-\chi (2t)$ and $\psi_{N}(t):=\psi (N^{-1}t)$. Then, $\sum_{N}\psi_{N}(t)=1$ whenever $t\neq 0$. 
We define frequency and modulation projections
\[
\widehat{P_{N}u}(\xi ):=\psi_{N}(\xi )\widehat{u}(\xi ),\ 
\widetilde{Q_{M}^{S}u}(\tau ,\xi ):=\psi_{M}(\tau -|\xi|^{4})\widetilde{u}(\tau ,\xi ).
\]
Furthermore, we define $Q_{\geq M}^{S}:=\sum_{N\geq M}Q_{N}^{S}$ and $Q_{<M}^{S}:=Id -Q_{\geq M}^{S}$. 

The rest of this paper is planned as follows.
In Section 2, we will give the definition and properties of the $U^{p}$ space and $V^{p}$ space. 
In Section 3, we will give the multilinear estimates which are main estimates to prove Theorems~\ref{wellposed_1} and \ref{large-wp}. 
In Section 4, we will give the proof of the well-posedness and the scattering (Theorem~\ref{wellposed_1}, Corollary~\ref{sccat}, and Theorem \ref{large-wp}). 
In Section 5, we will give the proof of Theorem~\ref{wellposed_2}. 
In Section 6, we will give the proof of Theorem~\ref{notC3}. 
%

\section{The $U^{p}$, $V^{p}$ spaces  and their properties \label{func_sp}}
In this section, we define the $U^{p}$ space and the $V^{p}$ space, 
and introduce the properties of these spaces which are proved by Hadac, Herr and Koch (\cite{HHK09}, \cite{HHK10}). 

We define the set of finite partitions $\ZZ$ as
\[
\ZZ :=\left\{ \{t_{k}\}_{k=0}^{K}|K\in \N , -\infty <t_{0}<t_{1}<\cdots <t_{K}\leq \infty \right\}
\]
and if $t_{K}=\infty$, we put $v(t_{K}):=0$ for all functions $v:\R \rightarrow L^{2}$. 
\begin{defn}\label{upsp}
Let $1\leq p <\infty$. For $\{t_{k}\}_{k=0}^{K}\in \ZZ$ and $\{\phi_{k}\}_{k=0}^{K-1}\subset L^{2}$ with 
$\sum_{k=0}^{K-1} \| \phi_{k} \| _{L^{2}}^{p}=1$ we call the function $a:\R\rightarrow L^{2}$ 
given by
\[
a(t)=\sum_{k=1}^{K}\ee_{[t_{k-1},t_{k})}(t)\phi_{k-1}
\]
a ``$U^{p}${\rm -atom}''. 
Furthermore, we define the atomic space 
\[
U^{p}:=\left\{ \left. u=\sum_{j=1}^{\infty}\lambda_{j}a_{j}
\right| a_{j}:U^{p}{\rm -atom},\ \lambda_{j}\in \C \ {\rm such\ that}\  \sum_{j=1}^{\infty}|\lambda_{j}|<\infty \right\}
\]
with the norm
\[
 \| u \| _{U^{p}}:=\inf \left\{\sum_{j=1}^{\infty}|\lambda_{j}|\left|u=\sum_{j=1}^{\infty}\lambda_{j}a_{j},\ 
a_{j}:U^{p}{\rm -atom},\ \lambda_{j}\in \C\right.\right\}.
\]
\end{defn}
\begin{defn}\label{vpsp}
Let $1\leq p <\infty$. We define the space of the bounded $p$-variation 
\[
V^{p}:=\{ v:\R\rightarrow L^{2}|\  \| v \| _{V^{p}}<\infty \}
\]
with the norm
\[
 \| v \| _{V^{p}}:=\sup_{\{t_{k}\}_{k=0}^{K}\in \ZZ}\left(\sum_{k=1}^{K} \| v(t_{k})-v(t_{k-1}) \| _{L^{2}}^{p}\right)^{1/p}.
\]
Likewise, let $V^{p}_{-, rc}$ denote the closed subspace of all right-continuous functions $v\in V^{p}$ with 
$\lim_{t\rightarrow -\infty}v(t)=0$, endowed with the same norm  $ \| \cdot  \| _{V^{p}}$.
\end{defn}
\begin{prop}[\cite{HHK09} Proposition\ 2.2,\ 2.4,\ Corollary\ 2.6]\label{upvpprop}
Let $1\leq p<q<\infty$. \\
{\rm (i)} $U^{p}$, $V^{p}$ and $V^{p}_{-, rc}$ are Banach spaces. \\ 
{\rm (ii)} For every $v\in V^{p}$, $\lim_{t\rightarrow -\infty}v(t)$ and $\lim_{t\rightarrow \infty}v(t)$ exist in $L^{2}$. \\
{\rm (iii)} The embeddings $U^{p}\hookrightarrow V^{p}_{-,rc}\hookrightarrow U^{q}\hookrightarrow L^{\infty}_{t}(\R ;L^{2}_{x}(\R^{d}))$ are continuous. 
\end{prop}
\begin{thm}[\cite{HHK09} Proposition\ 2,10,\ Remark\ 2.12]\label{duality}
Let $1<p<\infty$ and $1/p+1/p'=1$. 
If $u\in V^{1}_{-,rc}$ be absolutely continuous on every compact intervals, then
\[
 \| u \| _{U^{p}}=\sup_{v\in V^{p'},  \| v \| _{V^{p'}}=1}\left|\int_{-\infty}^{\infty}(u'(t),v(t))_{L^{2}(\R^{d})}dt\right|.
\]
\end{thm}
\begin{defn}
Let $1\leq p<\infty$. We define
\[
U^{p}_{S}:=\{ u:\R\rightarrow L^{2}|\ S(-\cdot )u\in U^{p}\}
\]
with the norm $ \| u \| _{U^{p}_{S}}:= \| S(-\cdot )u \| _{U^{p}}$, 
\[
V^{p}_{S}:=\{ v:\R\rightarrow L^{2}|\ S(-\cdot )v\in V^{p}_{-,rc}\}
\]
with the norm $ \| v \| _{V^{p}_{S}}:= \| S(-\cdot )v \| _{V^{p}}$.
\end{defn}
\begin{rem}
The embeddings $U^{p}_{S}\hookrightarrow V^{p}_{S}\hookrightarrow U^{q}_{S}\hookrightarrow L^{\infty}(\R;L^{2})$ hold for $1\leq p<q<\infty$
by {\rm Proposition~\ref{upvpprop}}. 
\end{rem}
\begin{prop}[\cite{HHK09} Corollary\ 2.18]\label{projest}
Let $1< p<\infty$. We have
\begin{align}
& \| Q_{\geq M}^{S}u \| _{L_{tx}^{2}}\lesssim M^{-1/2} \| u \| _{V^{2}_{S}},\label{highMproj}\\
& \| Q_{<M}^{S}u \| _{V^{p}_{S}}\lesssim  \| u \| _{V^{p}_{S}},\ \  \| Q_{\geq M}^{S}u \| _{V^{p}_{S}}\lesssim  \| u \| _{V^{p}_{S}},\label{Vproj}
\end{align}
\end{prop}
\begin{prop}[\cite{HHK09} Proposition\ 2.19]\label{multiest}
Let 
\[
T_{0}:L^{2}(\R^{d})\times \cdots \times L^{2}(\R^{d})\rightarrow L^{1}_{loc}(\R^{d})
\]
be a $m$-linear operator. Assume that for some $1\leq p, q< \infty$
\[
 \| T_{0}(S(\cdot )\phi_{1},\cdots ,S(\cdot )\phi_{m}) \| _{L^{p}_{t}(\R :L^{q}_{x}(\R^{d}))}\lesssim \prod_{i=1}^{m} \| \phi_{i} \| _{L^{2}(\R^{d})}.
\]
Then, there exists $T:U^{p}_{S}\times \cdots \times U^{p}_{S}\rightarrow L^{p}_{t}(\R ;L^{q}_{x}(\R^{d}))$ satisfying
\[
 \| T(u_{1},\cdots ,u_{m}) \| _{L^{p}_{t}(\R ;L^{q}_{x}(\R^{d}))}\lesssim \prod_{i=1}^{m} \| u_{i} \| _{U^{p}_{S}}
\]
such that $T(u_{1},\cdots ,u_{m})(t)(x)=T_{0}(u_{1}(t),\cdots ,u_{m}(t))(x)$ a.e.
\end{prop}
Now we refer the Strichartz estimate for the fourth order Schr\"odinger equation proved by Pausader.
We say that a pair $(p,q)$ is admissible if $2 \le p,q \le \infty$, $(p,q,d) \neq (2, \infty ,2)$, and
\[
\frac{2}{p} + \frac{d}{q} = \frac{d}{2}.
\]
\begin{prop}[\cite{P07} Proposition\ 3.1]\label{Stri_est}
Let $(p,q)$ and $(a,b)$ be admissible pairs. 
Then, we have
\[
\begin{split}
\| S(\cdot )\varphi  \| _{L_{t}^{p}L_{x}^{q}}&\lesssim  \| |\nabla|^{-2/p}\varphi  \| _{L^{2}_{x}},\\
\left\| \int_{0}^{t}S(t-t' )F(t')dt'\varphi  \right\| _{L_{t}^{p}L_{x}^{q}}&\lesssim  \| |\nabla|^{-2/p-2/a}F \| _{L^{a'}_{t}L^{b'}_{x}},
\end{split}
\]
where $a'$ and $b'$ are conjugate exponents of $a$ and $b$ respectively.
\end{prop}
Propositions \ref{multiest} and ~\ref{Stri_est} imply the following.
\begin{cor}\label{Up_Stri}
Let $(p,q)$ be an admissible pair. 
\begin{equation}\label{U_Stri}
 \| u \| _{L_{t}^{p}L_{x}^{q}}\lesssim  \| |\nabla|^{-2/p}u \| _{U_{S}^{p}},\ \ u\in U^{p}_{S}.
\end{equation}
\end{cor}
Next, we define the function spaces which will be used to construct the solution. 
We define the projections $P_{>1}$ and $P_{<1}$ as
\[
P_{>1}:=\sum_{N\ge 1}P_N,\ P_{<1}:=Id-P_{>1}. 
\]

\begin{defn}\label{YZ_space}
Let $s <0$.\\
{\rm (i)} We define $\dot{Z}^{s}:=\{u\in C(\R ; \dot{H}^{s}(\R^{d}))\cap U^{2}_{S}|\  \| u \| _{\dot{Z}^{s}}<\infty\}$ with the norm
\[
 \| u \| _{\dot{Z}^{s}}:=\left(\sum_{N}N^{2s} \| P_{N}u \| ^{2}_{U^{2}_{S}}\right)^{1/2}.
\]
{\rm (ii)} We define $Z^{s}:=\{u\in C(\R ; H^{s}(\R^{d})) |\  \| u \| _{Z^{s}}<\infty\}$ with the norm
\[
 \| u \| _{Z^{s}}:= \| P_{<1} u \| _{\dot{Z}^{0}}+ \| P_{>1} u \| _{\dot{Z}^{s}}. 
\]
{\rm (iii)} We define $\dot{Y}^{s}:=\{u\in C(\R ; \dot{H}^{s}(\R^{d}))\cap V^{2}_{S}|\  \| u \| _{\dot{Y}^{s}}<\infty\}$ with the norm
\[
 \| u \| _{\dot{Y}^{s}}:=\left(\sum_{N}N^{2s} \| P_{N}u \| ^{2}_{V^{2}_{S}}\right)^{1/2}.
\]
{\rm (iv)} We define $Y^{s}:=\{u\in C(\R ; H^{s}(\R^{d})) |\  \| u \| _{Y^{s}}<\infty\}$ with the norm
\[
 \| u \| _{Y^{s}}:= \| P_{<1} u \| _{\dot{Y}^{0}}+ \| P_{>1 }u \| _{\dot{Y}^{s}}.
\]
\end{defn}
%
%
\section{Multilinear estimate for $P_{4}(u,\overline{u})=\overline{u}^{4}$ in $1d$ \label{Multi_est}}
%
%
In this section, we prove multilinear estimates for the nonlinearity $\partial_{x}(\overline{u}^{4})$ in $1d$, which plays a crucial role in the proof of Theorem \ref{wellposed_1}.
\begin{lemm}\label{modul_est}
We assume that $(\tau_{0},\xi_{0})$, $(\tau_{1}, \xi_{1})$, $\cdots$, $(\tau_{4}, \xi_{4})\in \R\times \R^{d}$ satisfy 
$\sum_{j=0}^{4}\tau_{j}=0$ and $\sum_{j=0}^{4}\xi_{j}=0$. Then, we have 
\begin{equation}\label{modulation_est}
\max_{0\leq j\leq 4}|\tau_{j}-|\xi_{j}|^{4}|
\geq \frac{1}{5}\max_{0\leq j\leq 4}|\xi_{j}|^{4}. 
\end{equation}
\end{lemm}
\begin{proof}
By the triangle inequality, we obtain (\ref{modulation_est}).  
\end{proof}

\subsection{The homogeneous case}

\begin{prop}\label{HL_est_n}
Let $d=1$ and $0<T\leq \infty$. 
For a dyadic number $N_{1}\in 2^{\Z}$, we define the set $A_{1}(N_{1})$ as
\[
A_{1}(N_{1}):=\{ (N_{2},N_{3},N_{4})\in (2^{\Z})^{3}|N_{1}\gg N_{2}\geq N_{3} \geq N_{4}\}. 
\]
If $N_{0}\sim N_{1}$, then we have
\begin{equation}\label{hl}
\begin{split}
&\left|\sum_{A_{1}(N_{1})}\int_{0}^{T}\int_{\R}\left(N_{0}\prod_{j=0}^{4}P_{N_{j}}u_{j}\right)dxdt\right|\\
&\lesssim 
 \| P_{N_{0}}u_{0} \| _{V^{2}_{S}} \| P_{N_{1}}u_{1} \| _{V^{2}_{S}}\prod_{j=2}^{4} \| u_{j} \| _{\dot{Y}^{-1/2}}. 
\end{split}
\end{equation}
\end{prop}
\begin{proof} 
We define $u_{j,N_{j},T}:=\ee_{[0,T)}P_{N_{j}}u_{j}$\ $(j=1,\cdots ,4)$ and put 
$M:=N_{0}^{4}/5$. We decompose $Id=Q^{S}_{<M}+Q^{S}_{\geq M}$. 
We divide the integrals on the left-hand side of (\ref{hl}) into $10$ pieces of the form 
\begin{equation}\label{piece_form_hl}
\int_{\R}\int_{\R}\left(N_{0}\prod_{j=0}^{4}Q_{j}^{S}u_{j,N_{j},T}\right) dxdt
\end{equation}
with $Q_{j}^{S}\in \{Q_{\geq M}^{S}, Q_{<M}^{S}\}$\ $(j=0,\cdots ,4)$. 
By the Plancherel's theorem, we have
\[
(\ref{piece_form_hl})
= c\int_{\sum_{j=0}^{4}\tau_{j}=0}\int_{\sum_{j=0}^{4}\xi_{j}=0}N_{0}\prod_{j=0}^{4}\F[Q_{j}^{S}u_{j,N_{j},T}](\tau_{j},\xi_{j}),
\]
where $c$ is a constant. Therefore, Lemma~\ref{modul_est} implies that
\[
\int_{\R}\int_{\R}\left(N_{0}\prod_{j=0}^{m}Q_{<M}^{S}u_{j,N_{j},T}\right) dxdt=0.
\]
So, let us now consider the case that $Q_{j}^{S}=Q_{\geq M}^{S}$ for some $0\leq j\leq 4$. 

First, we consider the case $Q_{0}^{S}=Q_{\geq M}^{S}$.  
By the Cauchy-Schwartz inequality, we have
\[
\begin{split}
&\left|\sum_{A_{1}(N_{1})}\int_{\R}\int_{\R}\left(N_{0}Q_{\geq M}^{S}u_{0,N_{0},T}\prod_{j=1}^{4}Q_{j}^{S}u_{j,N_{j},T}\right)dxdt\right|\\
&\leq N_{0} \| Q_{\geq M}^{S}u_{0,N_{0},T} \| _{L^{2}_{tx}} \| Q_{1}^{S}u_{1,N_{1},T} \| _{L^{4}_{t}L^{\infty}_{x}}\prod_{j=2}^{4}\left\|\sum_{N_{j}\lesssim N_{1}}Q_{j}^{S}u_{j,N_{j},T}\right\|_{L^{12}_{t}L^{6}_{x}}. 
\end{split}
\]
Furthermore by (\ref{highMproj}) and $M\sim N_{0}^{4}$, we have
\[
 \| Q_{\geq M}^{S}u_{0,N_{0},T} \| _{L^{2}_{tx}}
\lesssim N_{0}^{-2} \| u_{0,N_{0},T} \| _{V^{2}_{S}}
\]
and by (\ref{U_Stri}) and $V^{2}_{S}\hookrightarrow U^{4}_{S}$, 
we have
\[
\begin{split}
 \| Q_{1}^{S}u_{1,N_{1},T} \| _{L_{t}^{4}L_{x}^{\infty}}
&\lesssim N_{1}^{-1/2} \| Q_{1}^{S}u_{1,N_{1},T} \| _{U^{4}_{S}}
\lesssim N_{1}^{-1/2} \| Q_{1}^{S}u_{1,N_{1},T} \| _{V^{2}_{S}}. 
\end{split}
\]
While by the Sobolev inequality, (\ref{U_Stri}), $V^{2}_{S}\hookrightarrow U^{12}_{S}$ and the Cauchy-Schwartz inequality for the dyadic sum
, we have
\begin{equation}\label{L12L6_est}
\begin{split}
\left\|\sum_{N_{j}\lesssim N_{1}}Q_{j}^{S}u_{j,N_{j},T}\right\|_{L^{12}_{t}L^{6}_{x}}
&\lesssim \left\| |\nabla |^{1/6}\sum_{N_{j}\lesssim N_{1}}Q_{j}^{S}u_{j,N_{j},T} \right\| _{L^{12}_{t}L^{3}_{x}}
\lesssim \left\| \sum_{N_{j}\lesssim N_{1}}Q_{j}^{S}u_{j,N_{j},T} \right\| _{V^{2}_{S}}\\
& \lesssim N_1^{1/2} \left( \sum _{N_j \lesssim N_1} N_j^{-1} \| u_{j,N_j,T} \|_{V^2_S}^2 \right) ^{1/2}
\lesssim N_{1}^{1/2} \| \ee_{[0,T)}u_{j} \| _{\dot{Y}^{-1/2}}
\end{split}
\end{equation}
for $2\leq j\leq 4$. Therefore, we obtain
\[
\begin{split}
&\left|\sum_{A_{1}(N_{1})}\int_{\R}\int_{\R}\left(N_{0}Q_{\geq M}^{S}u_{0,N_{0},T}\prod_{j=1}^{m}Q_{j}^{S}u_{j,N_{j},T}\right)dxdt\right|\\
&\lesssim 
 \| P_{N_{0}}u_{0} \| _{V^{2}_{S}} \| P_{N_{1}}u_{1} \| _{V^{2}_{S}}\prod_{j=2}^{4} \| u_{j} \| _{\dot{Y}^{-1/2}}
\end{split}
\]
by (\ref{Vproj}) since $ \| \ee_{[0,T)}u \| _{V^{2}_{S}}\lesssim  \| u \| _{V^{2}_{S}}$ for any $T\in (0,\infty]$. 
For the case $Q_{1}^{S}=Q_{\geq M}^{S}$ is proved in same way. 

Next, we consider the case $Q_{i}^{S}=Q_{\geq M}^{S}$ for some $2\le i \le 4$. 
By the H\"older inequality, we have
\[
\begin{split}
&\left|\sum_{A_{1}(N_{1})}\int_{\R}\int_{\R}\left(N_{0}Q_{\geq M}^{S}u_{i,N_{i},T}\prod_{\substack{0\le j\le 4\\ j\neq i}}Q_{j}^{S}u_{j,N_{j},T}\right)dxdt\right|\\
&\lesssim N_{0} \| Q_{0}^{S}u_{0,N_{0},T} \| _{L_{t}^{12}L_{x}^{6}} \| Q_{1}^{S}u_{1,N_{1},T} \| _{L_{t}^{4}L_{x}^{\infty}}\\
&\ \ \ \ \times \left\| \sum_{N_{i}\lesssim N_{1}}Q_{\geq M}^{S}u_{i,N_{i},T} \right\|_{L_{tx}^{2}}
\prod_{\substack{2\le j\le 4 \\ j\neq i}}\left\| \sum_{N_{j}\lesssim N_{1}}Q_{j}^{S}u_{j,N_{j},T} \right\| _{L_{t}^{12}L_{x}^{6}}. 
\end{split}
\]
By $L^{2}$ orthogonality and (\ref{highMproj}), we have
\begin{equation}\label{hi_mod_234}
\begin{split}
\left\| \sum_{N_{i}\lesssim N_{1}}Q_{\geq M}^{S}u_{i,N_{i},T}\right\| _{L_{tx}^{2}}
&\lesssim \left(\sum_{N_{2}} \| Q_{\geq M}^{S}u_{i,N_{i},T} \| _{L_{tx}^{2}}^{2}\right)^{1/2}\\
&\lesssim N_{1}^{-3/2} \| \ee_{[0,T)}u_{i} \| _{\dot{Y}^{-1/2}}
\end{split}
\end{equation}
since $M\sim N_{0}^{4}$. While, by the calculation way as the case $Q_{0}^{S}=Q_{\geq M}^{S}$, we have
\[
 \| Q_{0}^{S}u_{0,N_{0},T} \| _{L_{t}^{12}L_{x}^{6}}\lesssim  \| Q_{0}^{S}u_{0,N_{0},T} \| _{V^{2}_{S}},
\]
\[
 \| Q_{1}^{S}u_{1,N_{1},T} \| _{L_{t}^{4}L_{x}^{\infty}}\lesssim N_{1}^{-1/2} \| Q_{1}^{S}u_{1,N_{1},T} \| _{V^{2}_{S}}
\]
and
\[
\left\| \sum_{N_{j}\lesssim N_{1}}Q_{j}^{S}u_{j,N_{j},T} \right\|_{L_{t}^{12}L_{x}^{6}}
\lesssim N_{1}^{1/2} \| \ee_{[0,T)}u_{j} \| _{\dot{Y}^{-1/2}}. 
\]
Therefore, we obtain
\[
\begin{split}
&\left|\sum_{A_{1}(N_{1})}\int_{\R}\int_{\R}\left(N_{0}Q_{\geq M}^{S}u_{i,N_{i},T}\prod_{\substack{0\le j\le 4\\ j\neq i}}Q_{j}^{S}u_{j,N_{j},T}\right)dxdt\right|\\
&\lesssim 
 \| P_{N_{0}}u_{0} \| _{V^{2}_{S}} \| P_{N_{1}}u_{1} \| _{V^{2}_{S}}\prod_{j=2}^{4} \| u_{j} \| _{\dot{Y}^{-1/2}}
\end{split}
\]
by (\ref{Vproj}) since $ \| \ee_{[0,T)}u \| _{V^{2}_{S}}\lesssim  \| u \| _{V^{2}_{S}}$ for any $T\in (0,\infty]$. 
\end{proof}
\begin{prop}\label{HH_est}
Let $d=1$ and $0<T\leq \infty$. 
For a dyadic number $N_{2}\in 2^{\Z}$, we define the set $A_{2}(N_{2})$ as
\[
A_{2}(N_{2}):=\{ (N_{3}, N_{4})\in (2^{\Z})^{4}|N_{2}\geq N_{3}\geq N_{4}\}. 
\]
If $N_{0}\lesssim N_{1}\sim N_{2}$, then we have
\begin{equation}\label{hh}
\begin{split}
&\left|\sum_{A_{2}(N_{2})}\int_{0}^{T}\int_{\R}\left(N_{0}\prod_{j=0}^{4}P_{N_{j}}u_{j}\right)dxdt\right|\\
&\lesssim 
\frac{N_{0}}{N_{1}} \| P_{N_{0}}u_{0} \| _{V^{2}_{S}} \| P_{N_{1}}u_{1} \| _{V^{2}_{S}}N_{2}^{-1/2} \| P_{N_{2}}u_{2} \| _{V^{2}_{S}} \| u_{3} \| _{\dot{Y}^{-1/2}} \| u_{4} \| _{\dot{Y}^{-1/2}}. 
\end{split}
\end{equation}
\end{prop}
The proof of Proposition~\ref{HH_est} is quite similar as the proof of Proposition~\ref{HL_est_n}.

\subsection{The inhomogeneous case}
\begin{prop}\label{HL_est_n-inh}
Let $d=1$ and $0<T\leq 1$. 
For a dyadic number $N_{1}\in 2^{\Z}$, we define the set $A_{1}'(N_{1})$ as
\[
A_{1}'(N_{1}):=\{ (N_{2},N_{3},N_{4})\in (2^{\Z})^{3}|N_{1}\gg N_{2}\geq N_{3} \ge N_{4}, \, N_4 \le 1 \}. 
\]
If $N_{0}\sim N_{1}$, then we have
\begin{equation}\label{hl-inh}
\begin{split}
\left|\sum_{A_{1}'(N_{1})}\int_{0}^{T}\int_{\R}\left(N_{0}\prod_{j=0}^{4}P_{N_{j}}u_{j}\right)dxdt\right|
\lesssim T^{\frac{1}{6}} \| P_{N_{0}}u_{0} \| _{V^{2}_{S}} \| P_{N_{1}}u_{1} \| _{V^{2}_{S}} \prod_{j=2}^{4} \| u_{j} \| _{Y^{-1/2}}. 
\end{split}
\end{equation}
\end{prop}
\begin{proof}
We further divide $A_1'(N_1)$ into three pieces:
\begin{align*}
A_1'(N_1) & = \bigcup _{j=1}^3 A_{1,j}'(N_1), \\
A_{1,1}'(N_1) &:= \{ (N_{2},N_{3},N_{4}) \in  A_1'(N_1) : N_3 \ge 1 \} ,\\
A_{1,2}'(N_2) &:= \{ (N_{2},N_{3},N_{4}) \in  A_1'(N_1) : N_2 \ge 1 \ge N_3 \} ,\\
A_{1,3}'(N_2) &:= \{ (N_{2},N_{3},N_{4}) \in  A_1'(N_1) :  1 \ge N_2 \} .
\end{align*}
We define $u_{j,N_{j}}:=P_{N_{j}}u_{j}$, $u_{j,T}:=\ee_{[0,T)}u_{j}$ and $u_{j,N_{j},T}:=\ee_{[0,T)}P_{N_{j}}u_{j}$\ $(j=1,\cdots ,4)$.
We firstly consider the case $A_{1,1}'(N_1)$
In the case $T \le N_0^{-3}$, the H\"older inequality implies
\begin{align*}
& \left|\sum_{A_{1,1}'(N_{1})} \int_{0}^{T}\int_{\R}\left(N_{0}\prod_{j=0}^{4}P_{N_{j}}u_{j}\right)dxdt\right| \\
& \le N_0 \| \ee_{[0,T)}\|_{L^{2}_{t}}
\| u_{0,N_0} \| _{L_t^4 L_x^{\infty}} \| u_{1,N_1} \| _{L_t^4 L_x^{\infty}} 
\prod _{j=2}^3 \left\| \sum_{1\le N_j \le N_{1}} u_{j,N_j} \right\| _{L_t^{\infty} L_x^2} \| P_{<1} u_{4} \| _{L_t^{\infty} L_x^{\infty}}
\end{align*}
Furthermore by (\ref{U_Stri}) and $V^{2}_{S}\hookrightarrow U^{4}_{S}$, 
we have
\[
\begin{split}
\| u_{0,N_0} \| _{L_t^4 L_x^{\infty}} \| u_{1,N_1} \| _{L_t^4 L_x^{\infty}} 
&\lesssim N_{0}^{-1/2}\| u_{0,N_0} \| _{U^{4}_{S}}N_{1}^{-1/2} \| Q_{1}^{S}u_{1,N_{1}} \| _{U^{4}_{S}}\\
&\lesssim N_{0}^{-1} \| u_{0,N_{0}} \| _{V^{2}_{S}} \| u_{1,N_{1}} \| _{V^{2}_{S}}
\end{split}
\]
and by the Sobolev inequality, $V^{2}_{S}\hookrightarrow L^{\infty}_{t}L^{2}_{x}$ and the Cauchy-Schwartz inequality , we have
\[
\| P_{<1} u_{4} \| _{L_t^{\infty} L_x^{\infty}}\lesssim \| P_{<1} u_{4} \| _{L_t^{\infty} L_x^{2}}
\lesssim \left(\sum_{N\le 2}\|P_{N}P_{<1}u_{4}\|_{V^{2}_{S}}^{2}\right)^{1/2}
\le \|P_{<1}u_4\|_{\dot{Y}^{0}}
\]
While by $L^{2}$ orthogonality and $V^{2}_{S}\hookrightarrow L^{\infty}_{t}L^{2}_{x}$, we have
\[
\begin{split}
\left\| \sum_{1\le N_j \le N_{1}} u_{j,N_j} \right\| _{L_t^{\infty} L_x^2}
&\lesssim \left(\sum_{1\le N_j \le N_{1}} \| u_{j,N_{j}} \| _{V^{2}_{S}}^{2}\right)^{1/2}
\lesssim N_{0}^{1/2} \| P_{>1}u_{j} \| _{\dot{Y}^{-1/2}}
\end{split}
\]
Therefore, we obtain
\[
\begin{split}
&\left|\sum_{A_{1,1}'(N_{1})} \int_{0}^{T}\int_{\R}\left(N_{0}\prod_{j=0}^{4}P_{N_{j}}u_{j}\right)dxdt\right| \\
&\lesssim T^{1/2}N_0 \| u_{0,N_{0}} \| _{V^{2}_{S}} \| u_{1,N_{1}} \| _{V^{2}_{S}}\prod_{j=2}^{3}\| P_{>1}u_{j} \| _{\dot{Y}^{-1/2}}\|P_{<1}u_4\|_{\dot{Y}^{0}}
\end{split}
\]
and note that $T^{1/2}N_0\le T^{1/6}$.

In the case $T \ge N_0^{-3}$, we divide the integrals on the left-hand side of (\ref{hl}) into $10$ pieces of the form \eqref{piece_form_hl} in the proof of Proposition \ref{HL_est_n}.
Thanks to Lemma~\ref{modul_est}, let us consider the case that $Q_{j}^{S}=Q_{\geq M}^{S}$ for some $0\leq j\leq 4$.
First, we consider the case $Q_{0}^{S}=Q_{\geq M}^{S}$.  
By the same way as in the proof of Proposition \ref{HL_est_n} and using
\[
\|Q_{4}^{S}P_{<1}u_{4,T}\|_{L^{12}_{t}L^{6}_{x}}\lesssim \|Q_{4}^{S}P_{<1}u_{4,T}\|_{V^{2}_{S}}\lesssim \|P_{<1}u_{4,T}\|_{\dot{Y}^{0}}
\]
instead of (\ref{L12L6_est}), we obtain
\[
\begin{split}
&\left|\sum_{A_{1,1}'(N_{1})}\int_{\R}\int_{\R}\left(N_{0}Q_{\geq M}^{S}u_{0,N_{0},T}\prod_{j=1}^{4}Q_{j}^{S}u_{j,N_{j},T}\right)dxdt\right|\\
&\leq N_{0} \| Q_{\geq M}^{S}u_{0,N_{0},T} \| _{L^{2}_{tx}} \| Q_{1}^{S}u_{1,N_{1},T} \| _{L^{4}_{t}L^{\infty}_{x}}
\prod_{j=2}^{3} \left\|\sum_{1 \le N_{j}\lesssim N_{1}}Q_{j}^{S}u_{j,N_{j},T}\right\|_{L^{12}_{t}L^{6}_{x}} \|Q_{4}^{S}P_{<1}u_{4,T}\|_{L^{12}_{t}L^{6}_{x}}\\
& \lesssim N_0^{-\frac{1}{2}} \| P_{N_0} u_0 \| _{V^2_S} \| P_{N_1} u_1 \| _{V^2_S} \prod_{j=2}^{3} \left\| P_{>1} u_j \right\| _{\dot{Y}^{-1/2}} \| P_{<1} u_{4} \| _{\dot{Y}^0}
\end{split}
\]
and note that $N_0^{-1/2}\le T^{1/6}$. 
Since the cases $Q_j^S = Q_{\ge M}^S$ ($j=1,2,3$) are similarly handled, we omit the details here.

We focus on the case $Q_4^S = Q_{\ge M}^S$.
By the same way as in the proof of Proposition \ref{HL_est_n} and using
\[
\|Q_{\ge M}^{S}P_{<1}u_{4,T}\|_{L^{2}_{tx}}\lesssim N_{0}^{-2} \|P_{<1}u_{4,T}\|_{V^{2}_{S}}\lesssim N_{0}^{-2}\|P_{<1}u_{4,T}\|_{\dot{Y}^{0}}
\]
instead of (\ref{hi_mod_234}) with $j=4$, we obtain
\[
\begin{split}
&\left|\sum_{A_{1,1}'(N_{1})}\int_{\R}\int_{\R}\left(N_{0}Q_{\geq M}^{S}u_{4,N_{4},T}\prod_{j=0}^{3}Q_{j}^{S}u_{j,N_{j},T}\right)dxdt\right|\\
&\leq N_{0} \| u_{0,N_{0},T} \| _{L^{12}_{t}L_x^6} \| Q_{1}^{S}u_{1,N_{1},T} \| _{L^{4}_{t}L^{\infty}_{x}}
\prod_{j=2}^{3} \left\|\sum_{1 \le N_{j}\lesssim N_{1}}Q_{j}^{S}u_{j,N_{j},T}\right\|_{L^{12}_{t}L^{6}_{x}} 
\|Q_{\geq M}^{S} P_{<1}u_{4,T}\|_{L^{2}_{tx}}\\
& \lesssim N_{0}^{-1/2}\| P_{N_0} u_0 \| _{V^2_S}  \| P_{N_1} u_1 \| _{V^2_S} \prod_{j=2}^{3} \left\| P_{>1} u_j \right\| _{\dot{Y}^{-1/2}} \| P_{<1} u_4 \| _{\dot{Y}^0}
\end{split}
\]
and note that $N_0^{-1/2}\le T^{1/6}$.

We secondly consider the case $A_{1,2}'(N_1)$.
In the case $T \le N_0^{-3}$, the H\"older inequality implies
\[
\begin{split}
& \left|\sum_{A_{1,2}'(N_{1})} \int_{0}^{T}\int_{\R}\left(N_{0}\prod_{j=0}^{4}P_{N_{j}}u_{j}\right)dxdt\right| \\
& \le N_0 \| \ee_{[0,T)}\|_{L^{2}_{t}}
\| u_{0,N_0} \| _{L_t^4 L_x^{\infty}} \| u_{1,N_1} \| _{L_t^4 L_x^{\infty}} \left\| \sum _{1 \le N_2 \lesssim N_1} u_{2,N_2} \right\| _{L_t^{\infty} L_x^2}
\prod_{j=3}^{4}\| P_{<1} u_{j} \| _{L_t^{\infty} L_x^4} .
\end{split}
\]
By the same estimates as in the proof for the case $A_{1,1}'(N_1)$ and
\[
\| P_{<1} u_{j} \| _{L_t^{\infty} L_x^4}\lesssim \| P_{<1} u_{j} \| _{L_t^{\infty} L_x^{2}}
\lesssim \left(\sum_{N\le 2}\|P_{N}P_{<1}u_{j}\|_{V^{2}_{S}}^{2}\right)^{1/2}
\le \|P_{<1}u_j\|_{\dot{Y}^{0}}
\]
for $j=3,4$, we obtain
\[
\begin{split}
&\left|\sum_{A_{1,2}'(N_{1})} \int_{0}^{T}\int_{\R}\left(N_{0}\prod_{j=0}^{4}P_{N_{j}}u_{j}\right)dxdt\right| \\
&\lesssim T^{1/2}N_0^{1/2} \| u_{0,N_{0}} \| _{V^{2}_{S}} \| u_{1,N_{1}} \| _{V^{2}_{S}}\| P_{>1}u_{2} \| _{\dot{Y}^{-1/2}}\prod_{j=3}^{4}\|P_{<1}u_j\|_{\dot{Y}^{0}}
\end{split}
\]
and note that $T^{1/2}N_0^{1/2}\le T^{1/3}$. 


In the case $T \ge N_0^{-3}$, we divide the integrals on the left-hand side of (\ref{hl}) into $10$ pieces of the form \eqref{piece_form_hl} in the proof of Proposition \ref{HL_est_n}.
Thanks to Lemma~\ref{modul_est}, let us consider the case that $Q_{j}^{S}=Q_{\geq M}^{S}$ for some $0\leq j\leq 4$.
By the same argument as in the proof for the case $A_{1,1}'(N_1)$, we obtain 
\[
\begin{split}
&\left|\sum_{A_{1,2}'(N_{1})}\int_{\R}\int_{\R}\left(N_{0}Q_{\geq M}^{S}u_{0,N_{0},T}\prod_{j=1}^{4}Q_{j}^{S}u_{j,N_{j},T}\right)dxdt\right|\\
&\leq N_{0} \| Q_{\geq M}^{S}u_{0,N_{0},T} \| _{L^{2}_{tx}} \| Q_{1}^{S}u_{1,N_{1},T} \| _{L^{4}_{t}L^{\infty}_{x}} \left\|\sum_{1 \le N_{2}\lesssim N_{1}}Q_{2}^{S}u_{2,N_{2},T}\right\|_{L^{12}_{t}L^{6}_{x}} \prod_{j=3}^{4} \| Q_{j}^{S}P_{<1}u_{j,T}\|_{L^{12}_{t}L^{6}_{x}}\\
& \lesssim N_0^{-1} \| P_{N_0} u_0 \| _{V^2_S} | P_{N_1} u_1 \| _{V^2_S} \left\| P_{>1} u_2 \right\| _{\dot{Y}^{-1/2}} \prod _{j=3}^4 \| P_{<1} v_j \| _{\dot{Y}^0}
\end{split}
\]
if $Q_0 = Q_{\ge M}^S$ and 
\[
\begin{split}
&\left|\sum_{A_{1,2}'(N_{1})}\int_{\R}\int_{\R}\left(N_{4}Q_{\geq M}^{S}u_{4,N_{4},T}\prod_{j=0}^{3}Q_{j}^{S}u_{j,N_{j},T}\right)dxdt\right|\\
&\leq N_{0} \| u_{0,N_{0},T} \| _{L^{12}_{t}L_x^6} \| Q_{1}^{S}u_{1,N_{1},T} \| _{L^{4}_{t}L^{\infty}_{x}} \left\|\sum_{1 \le N_{2}\lesssim N_{1}}Q_{2}^{S}u_{2,N_{2},T}\right\|_{L^{12}_{t}L^{6}_{x}} \\
&\hspace{21ex}\times \|Q_{3}^{S} P_{<1}u_{3,T}\|_{L^{12}_{t}L^{6}_{x}} \| Q_{\geq M}^{S} P_{<1}u_{4,T}\|_{L^{2}_{tx}}\\
& \lesssim N_0^{-1} \| P_{N_0} u_0 \| _{V^2_S} \| P_{N_1} u_1 \| _{V^2_S} \left\| P_{>1} u_2 \right\| _{\dot{Y}^{\frac{1}{2}}} \prod_{j=3}^{4}\| P_{<1} u_j \| _{\dot{Y}^0}
\end{split}
\]
if $Q_4 = Q_{\ge M}^S$
Note that $N_0^{-1}\le T^{1/3}$. 
The remaining cases follow from the same argument as above.

We thirdly consider the case $A_{1,3}'(N_1)$.
In the case $T \le N_0^{-3}$, the H\"older inequality implies
\[
\begin{split}
& \left|\sum_{A_{1,3}'(N_{1})} \int_{0}^{T}\int_{\R}\left(N_{0}\prod_{j=0}^{4}P_{N_{j}}u_{j}\right)dxdt\right| \\
& \le N_0 \| \ee_{[0,T)}\|_{L^{2}_{t}}\| u_{0,N_0} \| _{L_t^4 L_x^{\infty}} \| u_{1,N_1} \| _{L_t^4 L_x^{\infty}}
\prod_{j=2}^{4} \| P_{<1}u_{2} \| _{L_t^{\infty} L_x^3}.
\end{split}
\]
By the same estimates as in the proof for the case $A_{1,1}'(N_1)$ and
\[
\| P_{<1} u_{j} \| _{L_t^{\infty} L_x^3}\lesssim \| P_{<1} u_{j} \| _{L_t^{\infty} L_x^{2}}
\lesssim \left(\sum_{N\le 2}\|P_{N}P_{<1}u_{j}\|_{V^{2}_{S}}^{2}\right)^{1/2}
\le \|P_{<1}u_j\|_{\dot{Y}^{0}}
\]
for $j=2, 3,4$, we obtain
\[
\begin{split}
&\left|\sum_{A_{1,3}'(N_{1})} \int_{0}^{T}\int_{\R}\left(N_{0}\prod_{j=0}^{4}P_{N_{j}}u_{j}\right)dxdt\right| 
\lesssim T^{1/2}\| u_{0,N_{0}} \| _{V^{2}_{S}} \| u_{1,N_{1}} \| _{V^{2}_{S}}\prod_{j=2}^{4}\| P_{<1}u_{j} \| _{\dot{Y}^{0}}. 
\end{split}
\]

In the case $T \ge N_0^{-3}$, we divide the integrals on the left-hand side of (\ref{hl}) into $10$ pieces of the form \eqref{piece_form_hl} in the proof of Proposition \ref{HL_est_n}.
Thanks to Lemma~\ref{modul_est}, let us consider the case that $Q_{j}^{S}=Q_{\geq M}^{S}$ for some $0\leq j\leq 4$.
By the same argument as in the proof for the case $A_{1,1}'(N_1)$, we obtain 
\[
\begin{split}
&\left|\sum_{A_{1,3}'(N_{1})}\int_{\R}\int_{\R}\left(N_{0}Q_{\geq M}^{S}u_{0,N_{0},T}\prod_{j=1}^{4}Q_{j}^{S}u_{j,N_{j},T}\right)dxdt\right|\\
&\leq N_{0} \| Q_{\geq M}^{S}u_{0,N_{0},T} \| _{L^{2}_{tx}} \| Q_{1}^{S}u_{1,N_{1},T} \| _{L^{4}_{t}L^{\infty}_{x}} 
\prod_{j=2}^{4} \|Q_{j}^{S}P_{<1}u_{j,T}\|_{L^{12}_{t}L^{6}_{x}}\\
& \lesssim N_0^{-3/2} \| P_{N_0} u_0 \| _{V^2_S} \| P_{N_1} u_1 \| _{V^2_S} \left\| P_{<1} u_2 \right\| _{Y^{-1/2}} \prod _{j=3}^4 \| P_{<1} v_j \| _{\dot{Y}^0} 
\end{split}
\]
if $Q_0 = Q_{\ge M}^S$ and
\[
\begin{split}
&\left|\sum_{A_{1,3}'(N_{1})}\int_{\R}\int_{\R}\left(N_{4}Q_{\geq M}^{S}u_{4,N_{4},T}\prod_{j=0}^{3}Q_{j}^{S}u_{j,N_{j},T}\right)dxdt\right|\\
&\leq N_{0} \| u_{0,N_{0},T} \| _{L^{12}_{t}L_x^6} \| Q_{1}^{S}u_{1,N_{1},T} \| _{L^{4}_{t}L^{\infty}_{x}} \prod _{j=2}^3 \|Q_{j}^{S} P_{<1}u_{j,T}\|_{L^{12}_{t}L^{6}_{x}} 
\|Q_{\geq M}^{S} P_{<1}u_{4,T}\|_{L^2_{tx}}\\
& \lesssim N_0^{-3/2} \| P_{N_0} u_0 \| _{V^2_S} \| P_{N_1} u_1 \| _{V^2_S} \prod _{j=2}^4 \left\| P_{<1} u_j \right\| _{Y^{0}}
\end{split}
\]
if $Q_4 = Q_{\ge M}^S$.
Note that $N_0^{-3/2}\le T^{1/2}$. 
The cases  $Q_j^S = Q_{\ge M}^S$ ($j=1,2,3$) are the same argument as above. 

\end{proof}

Furthermore, we obtain the following estimate.

\begin{prop}\label{HH_est-inh}
Let $d=1$ and $0<T\leq 1$. 
For a dyadic number $N_{2}\in 2^{\Z}$, we define the set $A_{2}'(N_{2})$ as
\[
A_{2}'(N_{2}):=\{ (N_{3}, N_{4})\in (2^{\Z})^{4}|N_{2}\geq N_{3}\ge N_{4} , \, N_4 \le 1 \}. 
\]
If $N_{0}\lesssim N_{1}\sim N_{2}$, then we have
\begin{equation}\label{hh-inh}
\begin{split}
&\left|\sum_{A_{2}'(N_{2})}\int_{0}^{T}\int_{\R}\left(N_{0}\prod_{j=0}^{4}P_{N_{j}}u_{j}\right)dxdt\right|\\
&\lesssim T^{\frac{1}{6}} \frac{N_{0}}{N_{1}} \| P_{N_{0}}u_{0} \| _{V^{2}_{S}} \| P_{N_{1}}u_{1} \| _{V^{2}_{S}}N_{2}^{-1/2} \| P_{N_{2}}u_{2} \| _{V^{2}_{S}} \| u_{3} \| _{Y^{-1/2}} \| u_{4} \| _{Y^{-1/2}}. 
\end{split}
\end{equation}
\end{prop}

Because the proof is similar as above, we skip the proof.

%
%
\section{Proof of well-posedness \label{pf_wellposed_1}}

\subsection{The small data case}

In this section, we prove Theorem~\ref{wellposed_1} and Corollary~\ref{sccat}. 
We define the map $\Phi_{T, \varphi}$ as
\[
\Phi_{T, \varphi}(u)(t):=S(t)\varphi -iI_{T}(u,\cdots, u)(t),
\] 
where
\[
I_{T}(u_{1},\cdots u_{4})(t):=\int_{0}^{t}\ee_{[0,T)}(t')S(t-t')\partial_{x}\left(\prod_{j=1}^{4}\overline{u_{j}(t')}\right)dt'.
\]
To prove the well-posedness of (\ref{D4NLS}) in $\dot{H}^{-1/2}$, we prove that $\Phi_{T, \varphi}$ is a contraction map 
on a closed subset of $\dot{Z}^{-1/2}([0,T))$. 
Key estimate is the following:
\begin{prop}\label{Duam_est}
Let $d=1$. For any $0<T<\infty$, we have
\begin{equation}\label{Duam_est_1}
 \| I_{T}(u_{1},\cdots u_{4}) \| _{\dot{Z}^{-1/2}}\lesssim \prod_{j=1}^{4} \| u_{j} \| _{\dot{Y}^{-1/2}}.
\end{equation}
\end{prop}
\begin{proof}
We decompose
\[
I_{T}(u_{1},\cdots u_{m})=\sum_{N_{1},\cdots ,N_{4}}I_{T}(P_{N_{1}}u_{1},\cdots P_{N_{4}}u_{4}).
\]
By symmetry, it is enough to consider the summation for $N_{1}\geq N_{2}\geq N_{3} \geq N_{4}$. We put
\[
\begin{split}
S_{1}&:=\{ (N_{1},\cdots ,N_{m})\in (2^{\Z})^{m}|N_{1}\gg N_{2}\geq N_{3} \geq N_{4}\}\\
S_{2}&:=\{ (N_{1},\cdots ,N_{m})\in (2^{\Z})^{m}|N_{1}\sim N_{2}\geq N_{3} \geq N_{4}\}
\end{split}
\]
and
\[
J_{k}:=\left\| \sum_{S_{k}}I_{T}(P_{N_{1}}u_{1},\cdots P_{N_{4}}u_{4}) \right\| _{\dot{Z}^{-1/2}}\ (k=1,2).
\]

First, we prove the estimate for $J_{1}$. By Theorem~\ref{duality} and the Plancherel's theorem, we have
\[
\begin{split}
J_{1}&\leq \left\{ \sum_{N_{0}}N_{0}^{-1}\left\| S(-\cdot )P_{N_{0}}\sum_{S_{1}}I_{T}(P_{N_{1}}u_{1},\cdots P_{N_{4}}u_{4})\right\|_{U^{2}}^{2}\right\}^{1/2}\\
&\lesssim \left\{\sum_{N_{0}}N_{0}^{-1}\sum_{N_{1}\sim N_{0}}
\left( \sup_{ \| u_{0} \| _{V^{2}_{S}}=1}\left|\sum_{A_{1}(N_{1})}\int_{0}^{T}\int_{\R}\left(N_{0}\prod_{j=0}^{4}P_{N_{j}}u_{j}\right)dxdt\right|\right)^{2}\right\}^{1/2}, 
\end{split}
\]
where $A_{1}(N_{1})$ is defined in Proposition~\ref{HL_est_n}. 
Therefore by Proposition~\ref{HL_est_n}, we have
\[
\begin{split}
J_{1}&\lesssim \left\{\sum_{N_{0}}N_{0}^{-1}\sum_{N_{1}\sim N_{0}}
\left( \sup_{ \| u_{0} \| _{V^{2}_{S}}=1} \| P_{N_{0}}u_{0} \| _{V^{2}_{S}} \| P_{N_{1}}u_{1} \| _{V^{2}_{S}}\prod_{j=2}^{4} \| u_{j} \| _{\dot{Y}^{-1/2}}\right)^{2}\right\}^{1/2}\\
&\lesssim 
\left(\sum_{N_{1}}N_{1}^{-1} \| P_{N_{1}}u_{1} \| _{V^{2}_{\Delta}}^{2}\right)^{1/2}
\prod_{j=2}^{4} \| u_{j} \| _{\dot{Y}^{-1/2}}\\
&=\prod_{j=1}^{4} \| u_{j} \| _{\dot{Y}^{-1/2}}.
\end{split}
\]

Next, we prove the estimate for $J_{2}$. By Theorem~\ref{duality} and the Plancherel's theorem, we have
\[
\begin{split}
J_{2}&\leq
\sum_{N_{1}}\sum_{N_{2}\sim N_{1}}\left(\sum_{N_{0}}N_{0}^{-1}\left\|S(-\cdot )P_{N_{0}}\sum_{A_{2}(N_{2})}I_{T}(P_{N_{1}}u_{1},\cdots P_{N_{4}}u_{4})\right\|_{U^{2}}^{2}\right)^{1/2}\\
&=\sum_{N_{1}}\sum_{N_{2}\sim N_{1}}\left(\sum_{N_{0}\lesssim N_{1}}N_{0}^{-1}
\sup_{ \| u_{0} \| _{V^{2}_{S}}=1}\left| \sum_{A_{2}(N_{2})}\int_{0}^{T}\int_{\R}\left(N_{0}\prod_{j=0}^{4}P_{N_{j}}u_{j}\right)dxdt\right|^{2}\right)^{1/2}, 
\end{split}
\]
where $A_{2}(N_{2})$ is defined in Proposition~\ref{HH_est}. 
Therefore by {\rm Proposition~\ref{HH_est}} and Cauchy-Schwartz inequality for the dyadic sum, we have
\[
\begin{split}
J_{2}&\lesssim
\sum_{N_{1}}\sum_{N_{2}\sim N_{1}}\left(\sum_{N_{0}\lesssim N_{1}}N_{0}^{-1}
\left(\frac{N_{0}}{N_{1}} \| P_{N_{1}}u_{1} \| _{V^{2}_{S}}N_{2}^{-1/2} \| P_{N_{2}}u_{2} \| _{V^{2}_{S}} \| u_{3} \| _{\dot{Y}^{-1/2}} \| u_{4} \| _{\dot{Y}^{-1/2}}\right)^{2}\right)^{1/2}\\
&\lesssim \left(\sum_{N_{1}}N_{1}^{-1} \| P_{N_{1}}u_{1} \| _{V^{2}_{S}}^{2}\right)^{1/2}
\left(\sum_{N_{2}}N_{2}^{-1} \| P_{N_{2}}u_{2} \| _{V^{2}_{S}}^{2}\right)^{1/2} \| u_{3} \| _{\dot{Y}^{-1/2}} \| u_{4} \| _{\dot{Y}^{-1/2}}\\
&= \prod_{j=1}^{4} \| u_{j} \| _{\dot{Y}^{s_{c}}}.
\end{split}
\]
\end{proof}
\begin{proof}[\rm{\bf{Proof of Theorem~\ref{wellposed_1}.}}]
For $r>0$, we define 
\begin{equation}\label{Zr_norm}
\dot{Z}^{s}_{r}(I)
:=\left\{u\in \dot{Z}^{s}(I)\left|\  \| u \| _{\dot{Z}^{s}(I)}\leq 2r \right.\right\}
\end{equation}
which is a closed subset of $\dot{Z}^{s}(I)$. 
Let $T>0$ and $u_{0}\in B_{r}(\dot{H}^{-1/2})$ are given. For $u\in \dot{Z}^{-1/2}_{r}([0,T))$, 
we have
\[
 \| \Phi_{T,u_{0}}(u) \| _{\dot{Z}^{-1/2}([0,T))}\leq  \| u_{0} \| _{\dot{H}^{-1/2}} +C \| u \| _{\dot{Z}^{-1/2}([0,T))}^{4}\leq r(1+ 16 Cr^{3})
\]
and
\[
\begin{split}
 \| \Phi_{T,u_{0}}(u)-\Phi_{T,u_{0}}(v) \| _{\dot{Z}^{-1/2}([0,T))}
&\leq C( \| u \| _{\dot{Z}^{-1/2}([0,T))}+ \| v \| _{\dot{Z}^{-1/2}([0,T))})^{3} \| u-v \| _{\dot{Z}^{-1/2}([0,T))}\\
&\leq 64Cr^{3} \| u-v \| _{\dot{Z}^{-1/2}([0,T))}
\end{split}
\]
by Proposition~\ref{Duam_est} and
\[
 \| S(\cdot )u_{0} \| _{\dot{Z}^{-1/2}([0,T))}\leq  \| \ee_{[0,T)}S(\cdot )u_{0} \| _{\dot{Z}^{-1/2}}\leq  \| u_{0} \| _{\dot{H}^{-1/2}}, 
\] 
where $C$ is an implicit constant in (\ref{Duam_est_1}). Therefore if we choose $r$ satisfying
\[
r <(64C)^{-1/3},
\]
then $\Phi_{T,u_{0}}$ is a contraction map on $\dot{Z}^{-1/2}_{r}([0,T))$. 
This implies the existence of the solution of (\ref{D4NLS}) and the uniqueness in the ball $\dot{Z}^{-1/2}_{r}([0,T))$. 
The Lipschitz continuously of the flow map is also proved by similar argument. 
\end{proof} 
Corollary~\ref{sccat} is obtained by the same way as the proof of Corollaty\ 1.2 in \cite{Hi}. 

\subsection{The large data case}

In this subsection, we prove Theorem \ref{large-wp}.
The following is the key estimate.

\begin{prop}\label{Duam_est-inh}
Let $d=1$. We have
\begin{equation}\label{Duam_est_1-inh}
 \| I_{1}(u_{1},\cdots u_{4}) \| _{\dot{Z}^{-1/2}} \lesssim  \prod_{j=1}^{4} \| u_{j} \| _{Y^{-1/2}}.
\end{equation}
\end{prop}

\begin{proof}
We decompose $u_j = v_j +w_j$ with $v_j = P_{>1}u_j \in \dot{Y}^{-1/2}$ and $w_j  = P_{<1} u_j \in \dot{Y}^0$. 
>From Propositions \ref{HL_est_n-inh}, \ref{HH_est-inh}, and the same way as in the proof of Proposition~\ref{Duam_est}, 
it remains to prove that
\[
\| I_{1}(w_{1},w_2,w_3,w_{4}) \| _{\dot{Z}^{-1/2}} \lesssim \prod_{j=1}^{4} \| u_{j} \| _{\dot{Y}^0}.
\]
By Theorem \ref{duality}, the Cauchy-Schwartz inequality, the H\"older inequality and the Sobolev inequality, we have
\[
\| I_{1}(w_{1},w_2,w_3,w_{4}) \| _{\dot{Z}^{-1/2}}
\lesssim \left\| \prod_{j=1}^{4}\overline{w_{j}} \right\|_{L^1([0,1];L^2)}
\lesssim \prod _{j=1}^4 \| w_j \| _{L_t^{\infty} L_x^2}
\lesssim \prod_{j=1}^{4} \| u_{j} \| _{\dot{Y}^{0}},
\]
which completes the proof.
\end{proof}

\begin{proof}[\rm{\bf{Proof of Theorem \ref{large-wp}}}]
Let $u_0 \in B_{\delta ,R}(H^{-1/2})$ with $u_0=v_0+w_0$, $v_0 \in \dot{H}^{-1/2}$, $w_0 \in L^2$.
A direct calculation yields
\[
\| S(t) u_0 \| _{Z^{-1/2}([0,1))} \le \delta +R.
\]
We start with the case $R=\delta = (4C+4)^{-4}$, where $C$ is the implicit constant in \eqref{Duam_est_1-inh}.
Proposition \ref{Duam_est-inh} implies that for $u \in Z^{-1/2}_r([0,1])$ with $r=1/(4C+4)$
\begin{align*}
\| \Phi_{1,u_{0}}(u) \| _{Z^{-1/2}([0,1))} &  \leq  \| S(t) u_0 \| _{Z^{-1/2}([0,1))} +C \| u \| _{Z^{-1/2}([0,1))}^{4} \\
& \leq 2r^4 + 16C r^4
= r^4 (16C+2)
\le r
\end{align*}
and
\begin{align*}
\| \Phi_{1,u_{0}}(u)-\Phi_{1,u_{0}}(v) \| _{Z^{-1/2}([0,1))}
&\leq C( \| u \| _{Z^{-1/2}([0,1))}+ \| v \| _{Z^{-1/2}([0,1))})^{3} \| u-v \| _{Z^{-1/2}([0,1))}\\
&\leq 64Cr^{3} \| u-v \| _{Z^{-1/2}([0,1))}
< \| u-v \| _{Z^{-1/2}([0,1))}
\end{align*}
if we choose $C$ large enough (namely, $r$ is small enough).
Accordingly, $\Phi_{1,u_{0}}$ is a contraction map on $\dot{Z}^{-1/2}_{r}([0,1))$.

We note that 
all of the above remains valid if we exchange $Z^{-1/2}([0,1))$ by the smaller space $\dot{Z}^{-1/2}([0,1))$ since $\dot{Z}^{-1/2}([0,1)) \hookrightarrow Z^{-1/2}([0,1))$ and the left hand side of \eqref{Duam_est_1-inh} is the homogeneous norm.

We now assume that $u_0 \in B_{\delta ,R}(H^{-1/2})$ for $R \ge \delta = (4C+4)^{-4}$.
We define $u_{0, \lambda}(x) = \lambda ^{-1} u_0 (\lambda ^{-1}x)$.
For $\lambda = \delta ^{-2} R^{2}$, we observe that $u_{0,\lambda} \in B_{\delta ,\delta}(H^{-1/2})$.
We therefore find a solution $u_{\lambda} \in Z^{-1/2}([0,1))$ with $u_{\lambda}(0,x) = u_{0,\lambda}(x)$.
By the scaling, we find a solution $u \in Z^{-1/2}([0, \delta ^8 R^{-8}))$.

Thanks to Propositions \ref{HL_est_n-inh} and \ref{HH_est-inh}, the uniqueness follows from the same argument as in \cite{HHK10}.
\end{proof}

%
%
%
\section{Proof of Theorem~\ref{wellposed_2}}\label{pf_wellposed_2}\kuuhaku
In this section, we prove Theorem~\ref{wellposed_2}. 
We only prove for the homogeneous case since the proof for the inhomogeneous case is similar. 
We define the map $\Phi_{T, \varphi}^{m}$ as 
\[
\Phi_{T, \varphi}^{m}(u)(t):=S(t)\varphi -iI_{T}^{m}(u,\cdots, u)(t),
\] 
where
\[
I_{T}^{m}(u_{1},\cdots u_{m})(t):=\int_{0}^{t}\ee_{[0,T)}(t')S(t-t')\partial \left(\prod_{j=1}^{m}u_{j}(t')\right)dt'.
\]
and the solution space $\dot{X}^{s}$ as
\[
\dot{X}^{s}:=C(\R;\dot{H}^{s})\cap L^{p_{m}}(\R;\dot{W}^{s+1/(m-1),q_{m}}),
\] 
where $p_{m}=2(m-1)$, $q_{m}=2(m-1)d/\{(m-1)d-2\}$ for $d \ge 2$ and $p_3=4$, $q_3=\infty$ for $d=1$. 
To prove the well-posedness of (\ref{D4NLS}) in $L^{2}(\R )$ or $H^{s_{c}}(\R^{d})$, we prove that $\Phi_{T, \varphi}$ is a contraction map 
on a closed subset of $\dot{X}^{s}$. 
The key estimate is the following:
\begin{prop}\label{Duam_est_g}
{\rm (i)}\ Let $d=1$ and $m=3$. For any $0<T<\infty$, we have
\begin{equation}\label{Duam_est_1d}
 \| I_{T}^{3}(u_{1},u_{2}, u_{3}) \| _{\dot{X^{0}}}\lesssim T^{1/2}\prod_{j=1}^{3} \| u_{j} \| _{\dot{X}^{0}}.
\end{equation}
{\rm (ii)}\ Let $d\ge 2$, $(m-1)d\ge 4$ and $s_c=d/2-3/(m-1)$ For any $0<T\le \infty$, we have
\begin{equation}\label{Duam_est_2}
 \| I_{T}^{m}(u_{1},\cdots, u_{m}) \| _{\dot{X^{s_c}}}\lesssim \prod_{j=1}^{m} \| u_{j} \| _{\dot{X}^{s_c}}.
\end{equation}
\end{prop}
\begin{proof}
{\rm (i)}\ By Proposition~\ref{Stri_est} with $(a,b)=\left( 4, \infty \right)$,
we get
\[
 \| I_{T}^{3}(u_{1},u_{2}, u_{3}) \| _{L^{\infty}_{t}L^{2}_{x}}
\lesssim \left\|\ee_{[0,T)} |\nabla |^{-1/2}\partial \left(\prod_{j=1}^{3}u_{j}\right)\right\|_{L^{4/3}_{t}L^{1}_{x}}
\]
and
\[
 \| |\nabla |^{1/2}I_{T}^{3}(u_{1},u_{2}, u_{3}) \| _{L^{4}_{t}L^{\infty}_{x}}
\lesssim \left\| \ee_{[0,T)}|\nabla |^{1/2-1/2-1/2}\partial \left(\prod_{j=1}^{3}u_{j}\right)\right\|_{L^{4/3}_{t}L^{1}_{x}}.
\]
Therefore, thanks to the fractional Leibniz rule (see \cite{CW91}), we have
\[
\begin{split}
\| I_{T}^{3}(u_{1},\cdots, u_{3}) \| _{\dot{X^{0}}}
& \lesssim \left\| \ee_{[0,T)}|\nabla |^{1/2}\prod_{j=1}^{3}u_{j}\right\|_{L^{4/3}_{t}L^{1}_{x}} \\
& \lesssim  \| \ee_{[0,T)}\|_{L^{2}_{t}}\| |\nabla |^{1/2}u_{i} \| _{L^{4}_{t}L^{\infty}_{x}}\prod_{\substack{1\le j\le 3\\ j\neq i}} \| u_{j} \| _{L^{\infty}_{t}L^{2}_{x}}\\
&\lesssim T^{1/2}\prod_{j=1}^{3} \| u_{j} \| _{\dot{X}^{0}}
\end{split}
\]
by the H\"older inequality. 
\\
{\rm (ii)}\ By Proposition~\ref{Stri_est} with 
\begin{equation}\label{admissible_ab}
(a,b)=\left( \frac{2(m-1)}{m-2}, \frac{2(m-1)d}{(m-1)d-2(m-2)}\right),
\end{equation}
we get
\[
 \| |\nabla |^{s_c}I_{T}^{m}(u_{1},\cdots u_{m}) \| _{L^{\infty}_{t}L^{2}_{x}}
\lesssim \left\| |\nabla |^{s_c-2/a}\partial \left(\prod_{j=1}^{m}u_{j}\right)\right\|_{L^{a'}_{t}L^{b'}_{x}}
\]
and
\[
 \| |\nabla |^{s_c+1/(m-1)}I_{T}^{m}(u_{1},\cdots u_{m}) \| _{L^{p_m}_{t}L^{q_m}_{x}}
\lesssim \left\| |\nabla |^{s_c+1/(m-1)-2/p_m-2/a}\partial \left(\prod_{j=1}^{m}u_{j}\right)\right\|_{L^{a'}_{t}L^{b'}_{x}}.
\]
Therefore, thanks to the fractional Leibniz rule (see \cite{CW91}), we have
\[
\begin{split}
\| I_{T}^{m}(u_{1},\cdots u_{m}) \| _{\dot{X^{s_c}}}
& \lesssim \left\| |\nabla |^{s_c+1/(m-1)}\prod_{j=1}^{m}u_{j}\right\|_{L^{a'}_{t}L^{b'}_{x}} \\
& \lesssim  \sum_{i=1}^{m} \| |\nabla |^{s_c+1/(m-1)}u_{i} \| _{L^{p_{m}}_{t}L^{q_{m}}_{x}}\prod_{\substack{1\le j\le m\\ j\neq i}} \| u_{j} \| _{L^{p_{m}}_{t}L^{(m-1)d}_{x}}\\
&\lesssim \sum_{i=1}^{m} \| |\nabla |^{s_c+1/(m-1)}u_{i} \| _{L^{p_{m}}_{t}L^{q_{m}}_{x}}\prod_{\substack{1\le j\le m\\ j\neq i}} \| |\nabla |^{s_{c}+1/(m-1)}u_{j} \| _{L^{p_{m}}_{t}L^{q_{m}}_{x}}\\
&\lesssim \prod_{j=1}^{m} \| u_{j} \| _{\dot{X}^{s_c}}
\end{split}
\]
by the H\"older inequality and the Sobolev inequality, where we used the condition $(m-1)d\ge 4$ which is equivalent to $s_{c}+1/(m-1)\ge 0$. 
\end{proof}
The well-posedness can be proved by the same way as the proof of Theorem~\ref{wellposed_1} and the scattering follows from 
that the Strichartz estimate because the $\dot{X}^{s_c}$ norm of the nonlinear part is bounded by the norm of the $L^{p_m}L^{q_m}$ space (see for example \cite[Section 9]{P07}).
%
%
%
\section{Proof of Theorem~\ref{notC3}}\label{pf_notC3}

In this section we prove the flow of (\ref{D4NLS}) is not smooth.
Let $u^{(m)}[u_0]$ be the $m$-th iteration of \eqref{D4NLS} with initial data $u_0$:
\[
u^{(m)}[u_0] (t,x) := -i \int _0^t e^{i(t-t') \Delta ^2} \partial  P_m( S(t') u_0, S(-t') \overline{u_0}) dt' .
\]

Firstly we consider the case $d=1$, $m=3$, $P_{3}(u,\overline{u})=|u|^{2}u$. 
For $N\gg 1$, we put
\[
f_{N} = N^{-s+1/2} \mathcal{F}^{-1}[ \ee _{[N-N^{-1}, N+N^{-1}]}]
\]
Let $u^{(3)}_{N}$ be the third iteration of (\ref{D4NLS}) for the data $f_{N}$.
Namely, 
\[
u^{(3)}_{N}(t,x) = u^{(3)}[f_N] (t,x)= -i \int _0^t e^{i(t-t') \partial _x ^4} \partial _x  \left( |e^{it' \partial _x^4} f_{N}| ^2 e^{it' \partial _x^4} f_{N} \right)(x) dt'.
\]
Note that $ \| f_{N} \| _{H^s} \sim 1$. 
Thorem~\ref{notC3} is implied by the following propositions.
\begin{prop}
If $s<0$, then for any $N\gg 1$, we have
\[
 \| u^{(3)}_{N} \| _{L^{\infty}([0,1]; H^s)} \rightarrow \infty
\]
as $N\rightarrow \infty$. 
\end{prop}
\begin{proof}
A direct calculation implies
\[
\widehat{u^{(3)}_{N}} (t, \xi ) =  e^{it \xi ^4} \xi \int _{\xi _1-\xi _2+\xi _3 =\xi} \int _0^t e^{it'(-\xi ^4 +\xi _1^4-\xi _2^4+\xi _3^4)} d t' \widehat{f_{N}}(\xi _1) \overline{\widehat{f_{N}}}(\xi _2) \widehat{f_{N}}(\xi _3) 
\]
and
\begin{equation} \label{modulation}
\begin{split}
&-(\xi _1-\xi _2+\xi _3)^4+\xi _1^4-\xi _2^4+\xi _3^4\\
&= 2 (\xi _1- \xi _2)(\xi _2-\xi _3) ( 2 \xi _1^2 +\xi _2^2+2\xi _3^2 -\xi _1 \xi _2 -\xi _2\xi _3 +3 \xi _3 \xi _1) .
\end{split}
\end{equation}
>From $\xi _j \in [N-N^{-1}, N+N^{-1}]$ for $j=1,2,3$, we get
\[
|-(\xi _1-\xi _2+\xi _3)^4+\xi _1^4-\xi _2^4+\xi _3^4|
\lesssim 1.
\]
We therefore obtain for sufficiently small $t>0$
\begin{align*}
|\widehat{u^{(3)}_{N}} (t,\xi ) |
& \gtrsim t N^{-3s+5/2} \left| \int _{\xi _1-\xi _2+\xi _3 =\xi} \ee _{[N-N^{-1}, N+N^{-1}]} (\xi _1) \ee _{[N-N^{-1}, N+N^{-1}]} (\xi _2) \ee _{[N-N^{-1}, N+N^{-1}]} (\xi _3) \right| \\
& \gtrsim t N^{-3s+1/2} \ee _{[N-N^{-1},N+N^{-1} ]} (\xi ) .
\end{align*}
Hence,
\[
\| u^{(3)}_{N} \| _{L^{\infty}([0,1]; H^s)} \gtrsim N^{-2s}.
\]
This lower bound goes to infinity as $N$ tends to infinity if $s<0$, which concludes the proof.
\end{proof}

Secondly, we show that absence of a smooth flow map for $d \ge 1$ and $m \ge 2$.
Putting
\[
g_N := N^{-s-d/2} \mathcal{F}^{-1}[ \ee _{[-N,N]^d}] ,
\]
we set $u_N^{(m)} := u^{(m)} [g_N]$.
Note that $\| g_N \| _{H^s} \sim 1$.
As above, we show the following.

\begin{prop}
If $s<s_c := d/2-3/(m-1)$ and $\partial =|\nabla |$ or $\frac{\partial}{\partial x_k}$ for some $1\le k\le d$, then for any $N \gg 1$, we have
\[
\| u_N^{(m)} \| _{L^{\infty}([0,1];H^s)} \rightarrow \infty
\]
as $N \rightarrow \infty$.
\end{prop}
\begin{proof}
We only prove for the case $\partial =|\nabla |$ since the proof for the case $\frac{\partial}{\partial x_k}$ is same.
Let
\[
\mathcal{A} := \{ (\pm _1, \dots , \pm _m) : \pm _j \in \{ +, - \} \, (j=1, \dots ,m) \} .
\]
Since $\mathcal{A}$ consists of $2^m$ elements, we write
\[
\mathcal{A} = \bigcup _{\alpha}^{2^m} \{ \pm ^{(\alpha )} \} ,
\]
where $\pm ^{( \alpha )}$ is a $m$-ple of signs $+$ and $-$.
We denote by $\pm _{j}^{(\alpha )}$ the $j$-th component of $\pm ^{(\alpha )}$.
A simple calculation shows that
\[
\widehat{u_N^{(m)}} (t,\xi) = |\xi | \sum _{\alpha =0}^{2^m} e^{it |\xi |^4} \int _{\xi = \sum _{j=1}^m \pm _j^{(\alpha)} \xi _j} \int _0^t e^{it' (-|\xi|^4 + \sum _{j=1}^m \pm _j^{(\alpha )} |\xi _j|^4)} dt' \prod _{j=1}^m \widehat{g_N} (\xi _j) .
\]
From
\[
\left| -|\xi|^4 + \sum _{j=1}^m \pm _j^{(\alpha )} |\xi _j|^4 \right| \lesssim N^4
\]
for $|\xi _j| \le N$ ($j= 1, \dots , m$), we have
\[
|\widehat{u_N^{(m)}} (t, \xi )|
\gtrsim |\xi | N^{-4} N^{-m(s+d/2)} N^{(m-1)d} \ee _{[-N.N]^d} (\xi )
\gtrsim N^{-3} N^{-m(s+d/2)} N^{(m-1)d} \ee _{[N/2.N]^d} (\xi )
\]
provided that $t \sim N^{-4}$.
Accordingly, we obtain
\[
\| u_N^{(m)} (N^{-4}) \| _{H^s} \gtrsim N^{-3} N^{-m(s+d/2)} N^{(m-1)d} N^{s+d/2}
\sim N^{-(m-1)s+(m-1)d/2-3} ,
\]
which conclude that $\limsup _{t \rightarrow 0} \| u^{(m)}_N(t) \| _{H^s} = \infty$ if $s<s_c$.
\end{proof}

\section*{Acknowledgment}

The work of the second author was partially supported by JSPS KAKENHI Grant number 26887017.

\end{document}